\documentclass[11pt]{amsart}

 \usepackage[utf8]{inputenc}
\usepackage[T1]{fontenc,url}
\usepackage{todonotes}

\usepackage[breakable,skins]{tcolorbox}

 \usepackage[left=3.66cm,right=3.66cm,bottom=3cm,top=3cm]{geometry}

\usepackage{microtype}
\usepackage{amsmath,amsthm,amsfonts,amssymb}
\usepackage{mathrsfs}
\usepackage{enumerate}
\usepackage{cite}
\usepackage[all,2cell]{xy} \SilentMatrices
\usepackage{pdftricks}
\usepackage{epsfig,color}
\usepackage{mathtools}
\usepackage{tikz-cd}
\usepackage{stmaryrd} 

\usepackage{framed}
\usepackage{color}
\definecolor{shadecolor}{gray}{0.875}
\definecolor{col}{RGB}{42, 95, 151}
\usepackage[colorlinks=true, linkcolor=col, citecolor=col, filecolor = col, menucolor = col, urlcolor = col]{hyperref}

\numberwithin{equation}{section}

\theoremstyle{plain}
\newtheorem{theorem}{Theorem}[section]
\newtheorem*{lemma*}{Lemma}
\newtheorem{lemma}[theorem]{Lemma}
\newtheorem*{theorem*}{Theorem}
\newtheorem{proposition}[theorem]{Proposition}
\newtheorem*{proposition*}{Proposition}
\newtheorem{corollary}[theorem]{Corollary}
\newtheorem*{corollary*}{Corollary}

\theoremstyle{definition}
\newtheorem{remark}[theorem]{Remark}
\newtheorem*{remark*}{Remark}
\newtheorem*{definition*}{Definition}
\newtheorem{definition}[theorem]{Definition}
\newtheorem*{example*}{Example}

\newtheorem{example}[theorem]{Example}
\newtheorem{question}{Question}
\newtheorem*{question*}{Question}
 \newtheorem{claim}[theorem]{Claim}

\def\min{\operatorname{min}}

\def\rank{\operatorname{rank}}

\def\c1{\operatorname{c_1}}
\def\c2{\operatorname{c_2}}
\def\Spec{\operatorname{Spec}}
\def\Proj{\operatorname{Proj}}

\def\sing{\operatorname{sing}}
\def\CC{{\mathbb C}}
\def\ZZ{{\mathbb Z}}

\def\PP{{\mathbb P}}

\def\E{{\mathscr E}}

\def\OO{{\mathcal O}}

\def\AA{{\mathbb A}}

\def\+{\oplus}                   
\def\*{\otimes}

\DeclareMathOperator{\Gr}{Gr}

\def\Pic{\operatorname{Pic}}
\def\Hom{\operatorname{Hom}}
\def\Sym{\operatorname{Sym}}

\DeclareMathOperator{\Br}{Br}

\DeclareMathOperator{\OG}{OG}
\DeclareMathOperator{\GO}{GO}
\DeclareMathOperator{\GL}{GL}
\DeclareMathOperator{\Ogroup}{O}
\DeclareMathOperator{\SO}{SO}
\DeclareMathOperator{\Lie}{Lie}

\DeclareMathOperator{\BGO}{BGO}

\DeclareMathOperator{\BSO}{BSO}

\newcommand{\langto}{\xrightarrow{\hspace*{0.4cm}}}

\makeatletter
\def\myrightarrow{{\setbox\z@\hbox{$\rightarrow$}\dimen0\ht\z@\multiply\dimen0 6\divide\dimen0 10\ht\z@\dimen0\box\z@}}
\def\myrightarrowfill@{\arrowfill@\relbar\relbar\myrightarrow}
\newcommand{\myxrightarrow}[2][]{\ext@arrow 0359\myrightarrowfill@{#1}{#2}}
\def\myleftarrow{{\setbox\z@\hbox{$\leftarrow$}\dimen0\ht\z@\multiply\dimen0 6\divide\dimen0 10\ht\z@\dimen0\box\z@}}
\def\myleftarrowfill@{\arrowfill@\myleftarrow\relbar\relbar}
\newcommand{\myxleftarrow}[2][]{\ext@arrow 3095\myleftarrowfill@{#1}{#2}}
\makeatother

\renewcommand\setminus{-}

\newcommand{\longsquiggly}{\xymatrix{{}\ar@{~>}[r]&{}}}

\usepackage{amssymb}
\usepackage[all]{xy}
\usepackage{hyperref}

\theoremstyle{definition}

\newtheorem{say}[theorem]{}

\newtheorem*{ack}{Acknowledgments}      

\newtheorem{defn-theorem}[theorem]{Definition--Theorem}  
\newtheorem{defn-lem}[theorem]{Definition--Lemma}  

\theoremstyle{remark}

\renewcommand{\c}[0]{{\mathbb C}}  

\renewcommand{\o}[0]{{\mathcal O}} 
\newcommand{\z}[0]{{\mathbb Z}}

\renewcommand{\a}[0]{{\mathbb A}}

\newcommand{\p}[0]{{\mathbb P}}
\newcommand{\f}[0]{{\mathbb F}}

\newcommand{\map}[0]{\dasharrow}
\newcommand{\qtq}[1]{\quad\mbox{#1}\quad}

\newcommand{\pic}[0]{\operatorname{Pic}}

\newcommand{\tors}[0]{\operatorname{tors}}

\newcommand{\grass}[0]{\operatorname{Grass}}

\newcommand{\cl}[0]{\operatorname{Cl}}

\newcommand{\tsum}[0]{\textstyle{\sum}}

\def\loccoh#1.#2.#3.#4.{H^{#1}_{#2}(#3,#4)}

\DeclareMathAlphabet{\mathchanc}{OT1}{pzc}%
                                {m}{it}

\usepackage[all]{xy}\xyoption{dvips}

\def\Kollar{Koll\'ar}

\makeatletter
\let\@wraptoccontribs\wraptoccontribs
\makeatother

\title[Fano varieties with torsion in the third cohomology group]{Fano varieties with torsion\\ in the third cohomology group}

\author[John Christian Ottem and Jørgen Vold Rennemo]{John Christian Ottem and Jørgen Vold Rennemo \vspace{0.06cm}\\with an appendix by János Koll\'ar}

 \address{Department of Mathematics,
 University of Oslo,
 Box 1053, Blindern,
 0316 Oslo, Norway}
 \email{johnco@math.uio.no}
 
 \address{Department of Mathematics,
 University of Oslo,
 Box 1053, Blindern,
 0316 Oslo, Norway}

 \email{jorgeren@math.uio.no}

\address{Princeton University, Princeton NJ 08544-1000, USA}

  \email{kollar@math.princeton.edu}
  
\date{\today}

\begin{document}

\maketitle

\vspace{-0.8cm}
 \begin{abstract}
We construct first examples of Fano varieties with torsion in their third cohomology group. The examples are constructed as double covers of linear sections of rank loci of symmetric matrices, and can be seen as higher-dimensional analogues of the Artin--Mumford threefold. As an application, we answer a question of Voisin on the coniveau and strong coniveau filtrations of rationally connected varieties.
 \end{abstract}
\def\iff{\ensuremath{\Longleftrightarrow}}
\thispagestyle{empty}
\def\tors{{\rm tors}}
\def\Tors{\operatorname{Tors}}






\section{Introduction}
If $X$ is a nonsingular complex projective variety, the torsion subgroup of the integral cohomology group $H^3(X,\ZZ)$ is an important stable birational invariant. It was introduced by Artin and Mumford in \cite{artinmumford}, where they used the invariant to show that a certain unirational threefold is not rational.

For rationality questions, perhaps the most interesting class of varieties is that of Fano varieties, that is, smooth varieties with ample anticanonical divisor. In dimension at most 2, these are all rational, with $H^3(X,\ZZ) = 0$.
In dimension 3, there are 105 deformation classes of Fano varieties \cite{fanobook,MM1,MM2}, and direct inspection shows that in each class the group $H^3(X,\ZZ)$ is torsion free. 
Beauville asked on MathOverflow whether the same statement holds for Fano varieties in all dimensions \cite{mathoverflow}.\footnote{An incorrect counterexample is proposed in the answer to \cite{mathoverflow}; see Section \ref{AMcase}.}
In this paper, we answer the question in the negative.
\begin{theorem}\label{mainthm1}
For each even $d\ge 4$, there is a $d$-dimensional Fano variety $X$ of Picard rank 1 with $H^3(X,\ZZ)=\ZZ/2$.
\end{theorem}
As a consequence, by \cite{Campana_1992, KMM}, the variety $X$ is rationally connected but not stably rational. We do not know if it is unirational.

The $X$ in the theorem is a complete intersection in a double cover of the space of rank $\le 4$ quadrics in $\mathbb P^{d/2+2}$.
The families of maximal linear subspaces of these quadrics give Brauer--Severi varieties over $X$, and via the isomorphism $\Br(X) \cong \Tors H^3(X,\ZZ)$, the associated Brauer class maps to the nonzero element in $H^3(X,\ZZ)$.
Our examples can be regarded as higher-dimensional analogues of the Artin--Mumford threefold from \cite{artinmumford}, whose construction is closely related to that of our $X$ (see Section \ref{isolatedsingcase}).

Starting in dimension 6, the Fano varieties we consider have a further exotic property.
\begin{theorem}\label{mainthm2}
When $d\ge 6$, the $d$-dimensional Fano variety $X$ from Theorem \ref{mainthm1} has the property that the coniveau and strong coniveau filtrations differ.
More precisely,
\begin{equation}\label{coniveaucontainment}
    0 = \widetilde N^1H^3(X,\ZZ) \subsetneqq N^1H^3(X,\ZZ) = H^3(X,\ZZ) \cong \ZZ/2.
\end{equation}
\end{theorem}

The two coniveau filtrations $\widetilde N^cH^l(X,\ZZ) \subseteq N^cH^l(X,\ZZ)$ of $H^l(X,\ZZ)$ were introduced in the paper \cite{Benoist_Ottem_2021}.
The subgroups of the filtrations contain the cohomology classes in $H^l(X,\ZZ)$ obtained via pushforward from smooth projective varieties (resp.~possibly singular projective varieties) of codimension at least $c$. In the case $c = 1$, $l = 3$, they are described as follows. 
The group $N^1H^3(X,\ZZ)$ consists of classes in $H^3(X,\ZZ)$ supported on some divisor of $X$. 
Its subgroup $\widetilde N^1H^3(X,\ZZ)$ consists of pushforwards  $f_*\beta$ of classes $\beta\in H^1(S,\ZZ)$ via proper maps $f:S\to X$ where $S$ is nonsingular of dimension $\dim X-1$.

An inequality of the two coniveau filtrations is particularly interesting for $c=1$ because for each $l \ge 0$, the quotient 
\begin{equation}\label{coniveauquotient}
N^1H^l(X,\ZZ)/\widetilde N^1H^l(X,\ZZ)
\end{equation}
is a stable birational invariant for smooth projective varieties \cite[Proposition~2.4]{Benoist_Ottem_2021}. 

While the examples of \cite{Benoist_Ottem_2021} show that this quotient can be non-zero in general, it is known to be zero for certain classes of varieties. Voisin \cite{voisinconiveauRC} proved that for a rationally connected threefold, any class in $H^3(X,\ZZ)$ modulo torsion lies in $\widetilde N^1H^3(X,\ZZ)$. 
Tian \cite[Theorem 1.23]{tian2022local} strenghtened this to show that $H^3(X,\ZZ)=\widetilde{N}^1H^3(X,\ZZ)$ for any rationally connected threefold. Theorem \ref{mainthm2} shows that the quotient \eqref{coniveauquotient} can be non-zero for rationally connected $X$ of higher dimension, answering a question of Voisin (see \cite[Question 3.1]{voisinconiveauRC} and \cite[Section 7.2]{Benoist_Ottem_2021}).



The paper is organised as follows. 
Section \ref{symmetricdeterminantal} begins with background on the geometry of symmetric determinantal loci and their double covers.
In Section \ref{symmetricasGIT}, we explain how these symmetric determinental loci and their double covers are GIT quotients of affine space by an action of an orthogonal similitude group. 
In Section \ref{mainexamplessection} (more specifically Definition \ref{defn:CompleteIntersectionX}), we define the main examples in Theorem \ref{mainthm1} as linear sections of the double covers of symmetric rank loci.

In Section \ref{cohomologycomputations}, we use the presentation of the double symmetric rank loci as GIT quotients to show that their smooth part has non-trivial torsion classes $\alpha\in H^3(X,\ZZ)$.
Taking a linear section and applying a generalised Lefschetz hyperplane theorem then proves Theorem \ref{mainthm1}, restated more precisely as Theorem \ref{thm:mainthm1PreciseVersion}.
In Section \ref{specialexamples}, we study some special examples appearing in our construction and compute their geometric invariants, in particular, the ``minimal'' example of a 4-dimensional Fano variety.

In the final Section \ref{coniveausection} we prove Theorem \ref{mainthm2}, restated precisely as Theorem \ref{thm:mainThm2PreciseVersion}.
The key point is that the mod 2 reduction of the generator $\alpha$ of $H^3(X,\ZZ)$ satisfies $\overline \alpha^2 \not= 0 \pmod 2$, which implies that $\alpha$ is not of strong coniveau 1 by a topological obstruction described in \cite{Benoist_Ottem_2021}.


We would like to thank N. Addington, O. Benoist, J. \Kollar, S. Schreieder, F. Suzuki and C. Voisin for useful discussions.
The work on this paper was begun at the Oberwolfach workshop \textit{Algebraic Geometry: Moduli Spaces, Birational Geometry and Derived Aspects} in the summer of 2022.
J.V.R.~is funded by the Research Council of Norway grant no.~302277.
\subsection{Notation}
We work over the complex numbers $\CC$. We use the notation for projective bundles where $\PP(\E)$ consists of lines in $\E$.

By a Fano variety we mean a nonsingular projective variety with ample anticanonical bundle.

\section{Symmetric determinantal loci and related varieties}\label{symmetricdeterminantal} Here we survey basic facts on symmetric determinal loci. Some of these are well known; we in particular follow the works of Hosono--Takagi \cite[Section 2]{hosonotakagi2015} and Tyurin \cite{tyurin1975intersections}. 

Let $V=\CC^{n}$. 
We identify $\PP(\Sym^2V^\vee)$ with the space of quadrics in $\PP(V)$ and let $Z_{r,n}\subset \PP(\Sym^2V^\vee)$ denote the subset of quadrics of rank $r$. 
$Z_{r,n}$ is a quasi-projective variety; its closure $\overline{Z}_{r,n}$ parameterizes the quadrics of rank $\le r$ and is defined by the vanishing of the $(r+1)\times (r+1)$-minors of a generic $n\times n$ symmetric matrix. These give a nested chain of subvarieties of $\PP(\Sym^2 V^\vee)$,
$$
\overline{Z}_{1,n}\subset \overline{Z}_{2,n}\subset \ldots \subset \overline{Z}_{n,n} = \PP(\Sym^2V^\vee),
$$where $\overline{Z}_{1,n}$ is the 2nd Veronese embedding of $\PP^{n-1}$, and $\overline{Z}_{n-1,n}$ is the degree $n$ hypersurface defined by the determinant.


\begin{proposition}
\label{thm:propertiesOfZ}
The variety $Z_{r,n}$ is irreducible of dimension
\begin{equation}\label{eq:dimZr}
\dim Z_{r,n} = rn - \tfrac12 r^2 + \tfrac 12 r - 1.
\end{equation}
The singular locus of $\overline{Z}_{r,n}$ is $\overline{Z}_{r-1,n}$, unless $r = 1$ or $r = n$, in which case $\overline Z_{r,n}$ is smooth.
\end{proposition}
This proposition can be checked using the incidence variety $\widetilde{Z}_{r,n}$ parameterizing $(n-r-1)$-planes contained in the singular loci of quadrics
$$
\widetilde {Z}_{r,n} = \bigl\{ ([L],Q) \, \big| \, \PP(L)\subset \sing(Q) \bigr\}\subset \Gr(n-r,V)\times \PP(\Sym^2V^\vee).
$$For $[L]\in \Gr(n-r,V)$, the fiber of the first projection $\pi_1$, can be identified with the space of quadrics in $\PP(V/L)\simeq \PP^{r-1}$, so $\pi_1$ is a $\PP^{r(r+1)/2-1}$-bundle over $\Gr(n-r,V)$. 
It follows that $\widetilde{Z}_{r,n}$ is nonsingular, and its dimension is given by \eqref{eq:dimZr}. Moreover, it is straightforward to check that the second projection gives a desingularization $\pi_2:\widetilde{Z}_{r,n}\to Z_{r,n}$. For the claim about the singular locus, see \cite[Section 2]{hosonotakagi2015}.


\begin{example}\label{ex:Vdim5}
For $n=5$, $\PP(\Sym^2V^\vee)\simeq \PP^{14}$, and there are 4 closed rank loci:
\begin{itemize}
\item $\overline{Z}_{4,5}$ is a quintic hypersurface defined by the determinant of the generic $5\times 5$-symmetric matrix.
\item $\overline{Z}_{3,5}$ is a codimension 3 subvariety of degree 20.
\item $\overline{Z}_{2,5}\simeq\Sym^2\PP^4$ is a codimension 6 subvariety of degree 35.
\item $\overline{Z}_{1,5}$ is the 2nd Veronese embedding of $\PP^4$ in $\PP^{14}$; it is a fourfold of degree $16$.
\end{itemize}
\end{example}

\subsection{Double covers}\label{sec:doublecovers}
We will only be interested in the case when the rank $r$ is even. In this case, we can define a double cover 
$$\sigma:W_{r,n}\langto  \overline{Z}_{r,n}$$
which is ramified exactly over the locus $\overline{Z}_{r-1,n}$, of codimension $n-r+1$ in $\overline{Z}_{r,n}$.
The construction is based on the classical fact  that for a quadric $Q$ of rank $r$ in $n$ variables, the variety of $(n-r/2-1)$-planes in $Q\subset \PP^{n-1}$ is isomorphic to the orthogonal Grassmannian $\OG(r/2,r)$, which has two connected components. 

The formal construction of $W_{r,n}$ from this observation starts with the incidence variety
\begin{equation}
\label{eqn:DefinitionOfIncidenceVariety}
U_{r,n}=\biggl\{(L,Q) \,\,|\,\, Q\in \overline{Z}_{r,n} \text{ and }\PP(L)\subset Q\biggr\}\subset \Gr(n-r/2,V)\times \overline{Z}_{r,n}.
\end{equation}
Taking the Stein factorisation of the projection $U_{r,n} \to \overline{Z}_{r,n}$ we get a new variety $W_{r,n}$ and morphisms 
\begin{equation}\label{defineUandW}
\eta: U_{r,n}\to W_{r,n} \text{ and } \sigma:W_{r,n}\to \overline{Z}_{r,n},
\end{equation}
where $\eta$ has connected fibres and $\sigma$ is finite.
The fibre of $\eta$ at a general point of $W_{r,n}$ is isomorphic to a connected component of $\OG(r/2,r)$.
The morphism $\sigma$ is a double cover, ramified exactly along $\overline{Z}_{r-1,n}$  (see \cite[Proposition 2.3]{hosonotakagi2015}).

For the remainder of the paper, we will let 
\[
H=\sigma^*\mathcal O_{\overline{Z}_{r,n}}(1)
\] be
the pullback of the polarization from $\PP(\Sym^2 V^\vee)$.
The basic geometric properties of $W_{r,n}$ are as follows.
\begin{proposition}\label{propertiesofWr}
$W_{r,n}$ has Gorenstein canonical singularities contained in $\sigma^{-1}(\overline{Z}_{r-2,n})$.
     It has Picard number 1 and its anticanonical divisor is 
    \begin{equation}\label{KW}
-K_{W_{r,n}}=\frac{rn}{2} H.
    \end{equation}In particular, $W_{r,n}$ is a singular Fano variety.

\end{proposition}
\begin{proof}
    See \cite[Proposition 2.5]{hosonotakagi2015}. 
\end{proof}

\begin{example}In the setting of Example \ref{ex:Vdim5}, $-K_{W_{4,5}}=10H$ follows because $Z_{4,5}\subset \PP^{14}$ is a quintic hypersurface and $\sigma$ is étale over $Z_{4,5}-Z_{2,5}$.
\end{example}


\subsection{(Double) symmetric determinantal loci as GIT quotients}\label{symmetricasGIT}
In this section, we explain how the varieties $\overline Z_{r,n}$ and $W_{r,n}$ can be presented as GIT quotients of affine spaces, which is a key ingredient in the cohomology computations needed in Theorems \ref{mainthm1} and \ref{mainthm2}.

Let $r$ be even, let $S = \CC^r$, and let $\omega_S \in \Sym^2 S^\vee$ be a nondegenerate quadratic form.
The {\em orthogonal similitude group} $\GO(S) \subset \GL(S)$ consists of the linear automorphisms of $S$ which preserve $\omega_S$ up to scaling.\footnote{The group $\GO(S)$ could more properly be denoted $\GO(S,\omega_S)$, but since the choice of $\omega_S$ does not matter, we omit it from the notation.}
In other words, an invertible linear map $\phi \colon S \to S$ lies in $\GO(S)$ if there exists a $\chi(\phi) \in \CC^*$ such that for all $v \in S$,
\[
\chi(\phi)\omega_S(v,v) = \omega_S(\phi(v),\phi(v)).
\]
The map $\chi \colon \GO(S) \to \CC^*$ defined by this relation is a group homomorphism, and 
we have an exact sequence
\[
1 \langto \Ogroup(S) \langto \GO(S) \overset{\chi}{\langto} \mathbb C^* \langto 1.
\]
The group $\GO(S)$ naturally acts on the orthogonal Grassmannian $\OG(r/2,S)$.
The variety $\OG(r/2,S)$ has two connected components, and the action of $\GO(S)$ on this two-element set gives an exact sequence
\[
1 \langto \GO(S)^\circ \langto \GO(S) \langto \mu_2 \langto 1
\]
where $\GO(S)^\circ$ is connected.
We further have $\SO(S) = \Ogroup(S) \cap \GO(S)^\circ$, and an exact sequence
\[
1 \langto \SO(S) \langto \GO(S)^\circ \overset{\chi}{\langto} \CC^* \langto 1.
\]

Consider now the affine space $\Hom(V,S) \simeq \mathbb A^{rn}$.
The group $\GO(S)$ acts on $\Hom(V,S)$ via
\begin{gather*}
\GO(S) \times \Hom(V,S) \to \Hom(V,S) \\
(\phi, f) \mapsto \phi \circ f.
\end{gather*}
We have a morphism of affine spaces
\[
\tau \colon \Hom(V,S) \to \Sym^2 V^\vee,
\]
defined by, for any $f \in \Hom(V,S)$ and $v,w \in V$,
\[
\tau(f)(v,w) = f^*\omega_S(v,w) = \omega_S(f(v), f(w)).
\]
Let $CZ_{r,n} \subset \Sym^2 V^\vee$ be the subset of $\Sym^2 V^\vee$ corresponding to quadratic forms of rank $r$, so that $Z_{r,n} = CZ_{r,n}/\CC^*$.
The set $\tau^{-1}(CZ_{r,n}) \subset \Hom(V,S)$ consists of the $f \colon V \to S$ such that $f^*\omega_S$ has rank $r$.
\begin{lemma}
\label{thm:freeAction}
    The group $\GO(S)$ acts freely on $\tau^{-1}(CZ_{r,n})$.
\end{lemma}
\begin{proof}
If $f \in \tau^{-1}(CZ_{r,n})$, then $f$ is surjective, so no element of $\GO(S)$ fixes $f$.
\end{proof}
\begin{lemma}
The group $\GO(S)^\circ$ acts freely on $\tau^{-1}(CZ_{r,n} \cup CZ_{r-1,n})$.
\end{lemma}
\begin{proof}
The previous lemma shows that $\GO(S)^\circ$ acts freely on $\tau^{-1}(CZ_{r,n})$.
So let $f \in \tau^{-1}(CZ_{r-1,n})$, and let $\phi \in \GO(S)$ be an element which fixes $f$.
We will show that $\phi$ is the identity.

Since $f^*\omega_S$ has rank $r-1$, we may find a basis $v_1,\ldots, v_n$ of $V$ such that 
\[
f^*(\omega_S)(v_i,v_j) = \omega_S(f(v_i),f(v_j)) = \begin{cases}1 \text{ if } 1 \le i = j \le r-1 \\ 0 \text{ otherwise.} \end{cases}
\]
The elements $f(v_1),\ldots, f(v_{r-1}) \in S$ are orthonormal, so we can choose a vector $e \in S$ such that $f(v_1),\ldots, f(v_{r-1}), e$ is an orthonormal basis for $S$.
Since $\phi$ fixes $f$, we have $\phi(f(v_i)) = f(v_i)$ for $1 \le i \le r-1$.
We have
\[
\omega_S(\phi(e),\phi(e)) = \omega_S(e,e).
\]
and, for each $i$,
\[
\omega_S(\phi(f(v_i)), \phi(e)) = \omega_S(f(v_i), e).
\]
This implies that $\phi(e) = \pm e$, and then the fact that $\phi \in \GO(S)^\circ$ forces $\phi(e) = e$.
This means that $\phi$ is the identity element.
\end{proof}

\begin{lemma}\label{BGlemma2}
The codimension of $\tau^{-1}(CZ_{r-2,n})$ in $\Hom(V,S)$ equals $n-r+2$.
\end{lemma}
\begin{proof}
Let $f \in \Hom(V,S)$.
The rank of $f^*\omega_S$ equals the rank of $\omega_S|_{f(V)}$.
Thus if $f^*\omega_S$ has rank $r-2$ we either have that $f(V)$ has dimension $r-2$ or that $f(V)$ has dimension $r-1$ and $\PP(f(V))$ is tangent to the quadric $V(\omega_S) \subset \PP(S)$.
The set of maps $f$ of rank $r-2$ has codimension $2(n-r+2)$, while the set of $f$ with rank $r-1$ has codimension $n-r+1$.
The further requirement that $\PP(f(V))$ is tangent to the quadric gives codimension $n-r+2$.
\end{proof}

The character $\chi$ of $\GO(S)$ induces a $\GO(S)$-linearisation of $\OO_{\Hom(V,S)}$ such that $x \in H^0(\Hom(V,S),\OO_{\Hom(V,S)})$ is $\GO(S)$-invariant if and only if for all $\phi \in \GO(S)$ and $f \in \Hom(V,S)$ we have
\[
x(\phi f) = \chi(\phi) x(f).
\]
Let $\Hom(V,S)^{ss} \subset \Hom(V,S)$ denote the associated GIT semistable locus and let $\Hom(V,S)^{us} = \Hom(V,S) \setminus \Hom(V,S)^{ss}$.
\begin{lemma}
    We have 
    \[
    \Hom(V,S)^{us} = \tau^{-1}(0),
    \]
    and an isomorphism
    \[
    \Hom(V,S)^{ss}\sslash \GO(S) \simeq \overline Z_{r,n}.
    \]
\end{lemma}
\begin{proof}
Let $R$ be the coordinate ring of $\Hom(V,S)$.
The GIT quotient $\Hom(V,S)^{ss}\sslash \GO(S)$ is given by $\Proj R^{\Ogroup(S)}$, where the ring $R^{\Ogroup(S)}$ is graded by the action of $\GO(S)$, an action which factors through $\chi \colon \GO(S) \to \CC^*$.
Any linear function $x$ on $\Sym^2 V^\vee$ defines an element $\tau^*(x) \in R^{\Ogroup(S)}$, and the first fundamental theorem of invariant theory for orthogonal groups says that these $\tau^*(x)$ generate $R^{\Ogroup(S)}$ \cite[p.~390]{procesi_lie_2007}.
This shows that $\Hom(V,S)^{us} = \tau^{-1}(0)$, and moreover that $\tau$ gives a closed embedding $\Hom(V,S)^{ss}\sslash \GO(S) \to \PP(\Sym^2V^\vee)$.
It is easy to see that its image is $\overline Z_{r,n}$.
\end{proof}

Thinking of $\chi$ as a character of $\GO(S)^\circ$, we get a $\GO(S)^\circ$-linearisation of $\OO_{\Hom(V,S)}$.
The associated GIT semistable locus in $\Hom(V,S)$ is the same as for the $GO(S)$-linearisation, since $\GO(S)^\circ$ has finite index in $\GO(S)$.
\begin{lemma}\label{WisGITquotient}
The GIT quotient $\Hom(V,S)^{ss}\sslash \GO(S)^\circ$ is isomorphic to $W_{r,n}$.
\end{lemma}
\begin{proof}
    Since $\GO(S)^\circ$ has finite index in $\GO(S)$, the morphism $\Hom(V,S)^{ss}\sslash \GO(S)^\circ \to \Hom(V,S)^{ss}\sslash \GO(S)$ is finite.
    Since $\Hom(V,S)^{ss}$ is smooth, $\Hom(V,S)^{ss}\sslash \GO(S)^\circ$ is normal \cite[p. 5]{GIT}.

    The open subset $\tau^{-1}(CZ_{r,n}) \sslash \GO(S)^\circ \subset \Hom(V,S)^{ss}\sslash \GO(S)^\circ$ is isomorphic to $\sigma^{-1}(Z_{r,n}) \subset W_{r,n}$ by the following construction.
    Fix an $r/2$-dimensional isotropic linear subspace $L \subset S$.
    Recall the variety $U_{r,n}$ from \eqref{eqn:DefinitionOfIncidenceVariety} and define a morphism $\gamma \colon \tau^{-1}(CZ_{r,n}) \to U_{r,n}$ by sending $f \in \tau^{-1}(CZ_{r,n})$ to the pair of the quadric
    \[
    Q = \{[v] \in \PP(V) \mid f^*\omega_S(v,v) = 0\}
    \]
    and the linear subspace $f^{-1}(L) \subset V$.
    Composing with $\eta \colon U_{r,n} \to W_{r,n}$ gives a morphism $\eta \circ \gamma \colon \tau^{-1}(CZ_{r,n}) \to W_{r,n}$.
    
    This morphism is $\GO(S)^\circ$-invariant, and one checks that it gives a bijection between the $\GO(S)^\circ$-orbits in $\tau^{-1}(CZ_{r,n})$ and the points of $\sigma^{-1}(Z_{r,n})$.
    Since $\sigma^{-1}(Z_{r,n})$ is smooth by Proposition \ref{propertiesofWr}, it follows from Zariski's main theorem that the induced morphism
    \[
    \psi \colon \tau^{-1}(CZ_{r,n}) \sslash \GO(S)^\circ \to \sigma^{-1}(Z_{r,n})
    \]
    is an isomorphism.

    The birational map $\psi$ fits in the following commutative diagram:
    \[
    \begin{tikzcd}
    \Hom(V,S)^{ss}\sslash \GO(S)^\circ \arrow[d] \arrow[r, "\psi", dashed] & W_{r,n} \arrow[d, "\sigma"] \\
    \Hom(V,S)^{ss}\sslash \GO(S) \arrow[r, "\simeq"] & \overline{Z}_{r,n}  
    \end{tikzcd}
    \]
    Let $L$ be the function field of $\Hom(V,S)^{ss}\sslash \GO(S)^\circ$, identified with the function field of $W_{r,n}$.
    Since these two varieties are normal and finite over $\overline{Z}_{r,n}$, they are both equal to the relative normalisation of $\overline{Z}_{r,n}$ in $\Spec L$, and so $\psi$ extends to an isomorphism of varieties.
\end{proof}



\begin{proposition}
\label{thm:localStructureOfW}
    Étale locally near a point $p \in \sigma^{-1}(Z_{r-2,n})$, the pair $(W_{r,n},p)$ is isomorphic to
    \[ 
    (C \times \AA^M, (0,0))
    \]
    where $C$ is the affine cone over the Segre embedding of $\PP^{n-r+1}\times \PP^{n-r+1}$, $0 \in C$ is the singular point, and $M = \dim W_{r,n} - \dim C$.
\end{proposition}
\begin{proof}
We use the isomorphism $\Hom(V,S)^{ss}\sslash \GO(S)^\circ \simeq W_{r,n}$.
Let $f \in \Hom(V,S)^{ss}$ be a point whose orbit maps to $\sigma^{-1}(Z_{r-2,n})$ under this isomorphism.
Then $f \in \tau^{-1}(CZ_{r-2,n})$, and we can choose a basis $v_1, \ldots, v_n$ for $V$ such that
\[
f^*\omega_S(v_i,v_j) = \omega_S(f(v_i),f(v_j)) = \begin{cases} 1 \text{ if } 1\le i = j \le r-2 \\ 
0 \text{ otherwise}\end{cases}
\]

This means that the elements $f(v_1),\ldots, f(v_{r-2})$ are orthonormal in $S$, and we extend this sequence to a basis of $S$ by adding vectors $e_1, e_2$ such that 
\begin{gather*}
\omega_S(e_i,e_j) = \delta_{ij} \\
\omega_S(e_i, f(v_j)) = 0 \text{ for all } i,j
\end{gather*}
Since the $f(v_{r-1}),\ldots, f(v_n)$ are orthogonal to each other and to each $f(v_i)$ with $1 \le i \le r-2$, the space 
\[
\langle f(v_{r-1}),\ldots, f(v_n)\rangle
\]
is an isotropic subspace of $\langle f(v_1),\ldots, f(v_{r-2}) \rangle^\perp = \langle e_1, e_2 \rangle$.

The isotropic subspaces of $\langle e_1, e_2 \rangle$ are $\langle e_1 \rangle$ and $\langle e_2 \rangle$.
Reordering the $e_i$, we may assume that $f(v_{r-1}),\ldots, f(v_n)$ are all contained in $\langle e_1 \rangle$.
After linearly transforming the $v_i$, we may assume that $f(v_{r-1}) = \gamma e_1$ for some $\gamma \in \CC$ and $f(v_i) = 0$ for $i \ge r$.
There is a subgroup $T \subset \GO(S)^\circ$, with $T \cong \CC^*$, consisting of elements $\phi_\lambda \in \GO(S)^\circ$, with $\lambda \in \CC^*$, defined by 
\begin{eqnarray*}
\phi_\lambda(f(v_i)) &=& f(v_i),\ \ \ \ \ i = 1,\ldots, r-2,\\
\phi_{\lambda}(e_1) &=& \lambda e_1 \\
\phi_{\lambda}(e_2) &=& \lambda^{-1}e_2.
\end{eqnarray*}
There are now two cases to consider:
If $\gamma \not= 0$, the stabiliser group of $f$ in $\GO(S)^\circ$ is trivial. 
The $GO(S)^\circ$-orbit of $f$ is not closed in $\Hom(V,S)^{ss}$, since $\lim_{\lambda \to 0}(\phi_\lambda f)$ is a point of $\Hom(V,S)^{ss}$.
If $\gamma = 0$, the stabiliser group of $f$ is $T$.
In this case the $GO(S)^\circ$-orbit of $f$ is closed in $\Hom(V,S)^{ss}$, since the orbit has minimal dimension among orbits mapping to $\sigma^{-1}(Z_{r-2,n})$.

Let us write $\AA^i(j)$ for a $T$-representation of dimension $i$ with weight $j$.
We then have an isomorphism of $T$-representations
\[
T_{\Hom(V,S),f} \cong \Hom(V,S) \cong \AA^n(1) \oplus \AA^n(-1) \oplus \AA^{n(r-2)}(0).
\]

The Luna étale slice theorem implies that étale locally near the orbit of $f$, the variety $\Hom(V,S)^{ss} \sslash \GO(S)^\circ$ is isomorphic to
\[
N_f \sslash T,
\]
where $N_f = T_{\Hom(V,S),f}/T_{\GO(S)^\circ f,f}$ is the normal space to the $\GO(S)^\circ$-orbit of $f$.
As $T$-representations
\[
T_{\GO(S)^\circ f,f} \cong \Lie(\GO(S)^\circ)/\Lie(T).
\]
Computing
\begin{gather*}
\Lie(\GO(S)^\circ) \cong \AA^{r-2}(1) \oplus \AA^{r-2}(-1) \oplus \AA^{\binom{r-2}{2}+2}(0)\\
\Lie(T) = \AA^1(0)
\end{gather*}
gives
\[
N_f \cong \AA^{n-r+2}(1) \oplus \AA^{n-r+2}(-1) \oplus \AA^M(0)
\]
for some $M$.
The quotient $N_f\sslash T$ is isomorphic to $C \times \AA^M$ with $C$ the cone over $\PP^{n-r+1}\times\PP^{n-r+1}$, so this completes the proof.
\end{proof}

\begin{corollary}
    If $r \ge 4$, the singular locus of $W_{r,n}$ is $\sigma^{-1}(\overline Z_{r-2,n})$.
    \label{singularlocusWr}
\end{corollary}
\begin{proof}
By Proposition~\ref{propertiesofWr}, $W_{r,n}$ is nonsingular away from $\sigma^{-1}(\overline Z_{r-2,n})$.
If $r \ge 4$, then $W_{r,n}$ is singular at every $p \in \sigma^{-1}(Z_{r-2,n})$, by Proposition~\ref{thm:localStructureOfW}, and the claim follows since the singular locus is closed.
\end{proof}

\subsection{Linear sections of double symmetric determinental loci}\label{mainexamplessection}
The varieties appearing in Theorems \ref{mainthm1} and \ref{mainthm2} will be constructed by taking general linear sections of the double cover $W_{r,n}$, i.e., complete intersections 
\begin{equation}\label{eq:ci}
    X=W_{r,n}\cap H_1\cap \cdots \cap H_c
\end{equation} where the $H_i$ are general divisors in $|H|$. In other words, $X$ is a ramified double cover of a linear section of $Z_{r,n}$,
$$
X\langto Z_{r,n}\cap L_1\cap \ldots \cap L_c
$$for hyperplanes $L_1,\ldots,L_c$ in $\PP(\Sym^2 V^\vee)$.

We are particularly interested in the case when $X$ is also a Fano variety. This can happen only when $r < 6$:

\begin{lemma}\label{ikkeFano}
    Let $X$ denote a general linear section of  $W_{r,n}$. If $6\le r\le n$, then either $X$ is singular, or $K_X$ is base-point free.
\end{lemma}

\begin{proof}
Write $X$ as in \eqref{eq:ci} for divisors $H_i\in |H|$. As the $H_i$ are general, and $W_{r,n}$ is Gorenstein with canonical singularities, it follows that the same holds for $X$. By Proposition \ref{propertiesofWr} and adjunction, the canonical divisor is given by
$$
K_X=\left(c-\frac{rn}2\right)H.
$$Therefore, if $c>rn/2$, $X$ is of general type, and for $c=rn/2$, it is Calabi--Yau. If $c<rn/2$, we note that
$$
\dim \sigma^{-1}(Z_{r-2,n})-c\ge \dim Z_{r-2,n}-rn/2+1=(n-r)(r-4)/2+r/2-3
$$which is non-negative for our choices of $r$ and $n$. This means that $X$ meets the singular locus of $W_{r,n}$, and hence it must be singular.
\end{proof}

By Lemma \ref{ikkeFano}, we obtain Fano varieties as linear sections of $W_{r,n}$ only when $r=2$ or $r=4$. 
The case $r = 2$ gives $W_{2,n}=\PP^{n-1}\times \PP^{n-1}$, and many linear sections of $\PP^{n-1} \times \PP^{n-1}$ are indeed Fano, but these varieties do not have interesting cohomology groups from the point of view of this paper.

We therefore focus on the case $r = 4$.
In this case the existence of the double cover $\sigma:W_{4,n}\to {\overline Z}_{4,n}$ is explained as follows. A smooth quadric surface in $\PP^3$ contains two families of lines; thus a quadric of rank 4 in $n$ variables contains two families of $(n-2)$-planes, each parameterised by a $\PP^1$. Thus $W_{4,n}$ parameterises quadrics plus a choice of one of the two families.

The dimensions of the first few rank loci $Z_i$ are given by
$$
\dim Z_{4,n}=4n-7, \,\, \dim Z_{3,n}=3n-4, \,\, \dim Z_{2,n}=2n-2,  \,\,\dim Z_{1,n} =n-1.
$$
By Corollary~\ref{singularlocusWr}, the double cover $W_{4,n}$ is singular along $\sigma^{-1}(\overline Z_{2,n})$, which has codimension $2n-5$ in $W_{4,n}$. By \eqref{KW}, the canonical divisor of $W_{4,n}$ equals
$$
K_{W_{4,n}}=-2nH.
$$
\begin{definition}
\label{defn:CompleteIntersectionX}
    Given $n \ge 4$ and $c \ge 0$, let $X_{n,c}$ be a general complete intersection 
\begin{equation}\label{eq:definerX}
    X_{n,c}=W_{4,n}\cap H_1\cap \cdots \cap H_{c}.
\end{equation}
\end{definition}
The varieties $X$ in Theorems \ref{mainthm1} and \ref{mainthm2} are $X_{n,2n-1}$ with $n \ge 5$ and $n \ge 6$, respectively.

\section{Cohomology computations}\label{cohomologycomputations}
Let $X_{n,c}^{\mathrm{sm}}$ be the smooth part of $X_{n,c}$.  In this section we compute the low degree cohomology of $X_{n,c}^{\mathrm{sm}}$. 
In Proposition~\ref{prop:cohomologyofGBO} we compute the low degree cohomology of $\BGO(4)^\circ$, 
and in Proposition~\ref{thm:compareBGOToX} we show that this agrees with low degree cohomology of $X^{\mathrm{sm}}_{n,c}$.
We summarise the consequences for the cohomology of $X_{n,c}^{\mathrm{sm}}$ in Corollary~\ref{thm:summaryOfHX}. 
In order to prove Theorem \ref{mainthm1}, we want a non-zero 2-torsion cohomology class of degree 3, and for Theorem \ref{mainthm2}, the class should furthermore have a non-zero square modulo 2 (this will be explained in Proposition \ref{maintopologicalobstruction}).

\subsection{Cohomology of $\BSO(4)$}
The cohomology rings with integer coefficients of the classifying spaces $\BSO(n)$ were computed by Brown \cite{brown} and Feshbach \cite{feshbach}. For $n=4$, the ring is given by
$$
H^*(\BSO(4),\ZZ) = \ZZ[\nu,e,p]/(2\nu),
$$where $e$ is the Euler class (of degree 4), $p$ is the Pontrjagin class (degree 4), and $\nu$ is a 2-torsion class of degree 3. Thus the low-degree cohomology groups of $\BSO(4)$ are given by
\begin{center}
\begin{tabular}{c|c|c|c|c|c|c}
   $H^0$ & $H^1$  & $H^2$ & $H^3$ & $H^4$ & $H^5$ & $H^6$\\ \hline
        $\ZZ$&  0 & 0 & $\ZZ/2 \cdot \nu$ & $\ZZ  p\oplus \ZZ e$ & 0 & $\ZZ/2 \cdot \nu^2$ \\
\end{tabular}
\end{center}The cohomology ring of $\BSO(4)$ with $\ZZ/2$-coefficients is given by
\begin{equation}\label{BSOring}
H^*(\BSO(4),\ZZ/2) = \ZZ/2 [w_2,w_3,w_4].
\end{equation}where $w_2,w_3,w_4 \in H^*(\BSO(4),\ZZ/2)$ denote the Stiefel--Whitney classes \cite{milnor}.

The class $\nu$ is equal to $\beta_\ZZ(w_2)\in H^3(\BSO(4),\ZZ)$, where $\beta_\ZZ$ is the Bockstein homomorphism associated with
$$
0\langto \ZZ\langto \ZZ \langto \ZZ/2\langto 0.
$$Moreover, the mod 2 reduction of $\nu$ is given by $w_3$ \cite[p. 97]{hatcher}. Therefore, the mod 2 reduction of $\nu^2$ equals
\[
w_3^2\neq 0 \, \text{ in } H^6(\BSO(4),\ZZ/2).
\]

\subsection{Cohomology of $\BGO(4)^\circ$}\label{BGOsection}
We next compute the low degree cohomology of $\BGO(4)^\circ$.
\begin{proposition}\label{prop:cohomologyofGBO}
We have
\[
H^1(\BGO(4)^\circ,\ZZ) = 0,\, H^2(\BGO(4)^\circ,\ZZ) = \ZZ, \,H^3(\BGO(4)^\circ,\ZZ) = \ZZ/2.
\]
Moreover, if $0 \not= \gamma \in H^3(\BGO(4)^\circ,\ZZ)$, the mod 2 reduction of $\gamma^2$ is non-zero.
\end{proposition}
\begin{proof}
We use the exact sequence 
$$
1\langto \SO(4)\langto \GO(4)^\circ \langto \CC^*\langto 1.
$$This gives a  fibre bundle $\pi:\BSO(4) \to \BGO(4)^\circ$ with fiber $\CC^*$. The associated Gysin sequence for this circle bundle takes the form
$$
\cdots \to H^i(\BGO(4)^\circ,\ZZ) \xrightarrow{\pi^*} H^i(\BSO(4),\ZZ) \xrightarrow{\pi_*} H^{i-1}(\BGO(4)^\circ,\ZZ) \to \cdots
$$
Using the fact that $H^0(\BSO(4),\ZZ) \to H^0(\BGO(4)^\circ,\ZZ)$ is an isomorphism and 
\[
H^1(\BSO(4),\ZZ) = H^2(\BSO(4),\ZZ) = 0,
\]
this sequence gives that
\[
H^1(\BGO(4)^\circ,\ZZ) = 0
\]
and
\[
H^2(\BGO(4)^\circ,\ZZ) = \ZZ.
\]
Since $H^3(\BSO(4),\ZZ) = \ZZ/2$, the map $H^3(\BSO(4),\ZZ) \to H^2(\BGO(4)^\circ,\ZZ)$ is 0, and so $H^3(\BGO(4),\ZZ) = \ZZ/2$.
If $0 \not= \gamma \in H^3(\BGO(4)^\circ,\ZZ)$, then $\pi^*(\gamma^2) = \nu^2$, which has reduction $w_3^2 \not= 0$ modulo 2.
It follows that $\gamma^2$ has non-zero reduction modulo 2.
\end{proof}


\subsection{Cohomology of hyperplane sections of $W_{r,n}^{\mathrm{sm}}$}
\label{sec:cohomologyOfHyperplaneSections}
Let $S$ be a quadratic $r$-dimensional vector space and $L \subseteq \PP(\Sym^2V^\vee)$ is a codimension $c$ linear subspace.
We analyse the natural homomorphism
\begin{equation}
\label{eqn:pullbackHomomorphism}
H^*(\BGO(S)^\circ,\ZZ) \to H^*(L \times_{\PP(\Sym^2 V^\vee)} W^{\mathrm{sm}}_{r,n},\ZZ)
\end{equation}
and show that it is an isomorphism in low degrees.
To define the homomorphism, begin with the pullback maps
\[
H^*(\BGO(S)^\circ,\ZZ) \overset{\simeq}{\to}H^*_{\GO(S)^\circ}(\Hom(V,S), \ZZ) \to H^*_{\GO(S)^\circ}(\Hom(V,S) \setminus \tau^{-1}(CZ_{r-2,n}),\ZZ),
\]
with $\tau$ and $CZ_{r-2,n}$ as defined in Section \ref{symmetricasGIT}.
By Lemma \ref{thm:freeAction} and Corollary \ref{singularlocusWr}, the variety $W_{r,n}^{\mathrm{sm}}$ is isomorphic to $(\Hom(V,S) \setminus \tau^{-1}(CZ_{r-2,n})/\GO(S)^\circ$, where the group action is free, so we get an isomorphism
\[
H^*_{\GO(S)^\circ}(\Hom(V,S) \setminus \tau^{-1}(CZ_{r-2,n}),\ZZ) \overset{\simeq}{\to} H^*(W_{r,n}^{\mathrm{sm}},\ZZ).
\]
Finally, we have the pullback homomorphism
\[
H^*(W_{r,n}^{\mathrm{sm}},\ZZ) \to H^*(L \times_{\PP(\Sym^2 V^\vee)} W_{r,n}^{\mathrm{sm}},\ZZ),
\]
and composing these maps gives \eqref{eqn:pullbackHomomorphism}.

\begin{lemma}\label{BGlemma1}
Let $G$ be an algebraic group on an affine space $\AA^N$.
Let $Z \subset \AA^N$ be a closed, $G$-invariant subset of codimension $c$, and let $U = \AA^N \setminus Z$.
Then the natural homomorphisms
\begin{gather*}
H^l_G(\mathrm{pt},\ZZ) \to H^l_G(U, \ZZ) \\
H^l_G(\mathrm{pt},\ZZ/2) \to H^l_G(U, \ZZ/2)
\end{gather*}
are isomorphisms for $l < 2c-1$, and injective for $l = 2c-1$.
\end{lemma}
\begin{proof}
The Leray--Serre spectral sequence for equivariant cohomology \cite[p. 501]{McCleary} has $E_2$-page $H^{i}_G(\mathrm{pt}, H^j(U))$ and converges to $H^{i + j}_G(U)$.
Since $H^j(U) = 0$ for $0 < j \le 2c - 2$, there are no non-trivial differentials whose domain is of degree $(i,j)$ with $i + j \le 2c-2$.
The claim of the theorem follows from this.
\end{proof}

\begin{lemma}
\label{thm:cohomologyOfComplement}
    The homomorphisms
    \begin{gather*}
    H^i_{\GO(S)^\circ}(\mathrm{pt},\ZZ) \to H^i_{\GO(S)^\circ}(\Hom(V,S) \setminus \tau^{-1}(CZ_{r-2,n}),\ZZ) \\
    H^i_{\GO(S)^\circ}(\mathrm{pt},\ZZ/2) \to H^i_{\GO(S)^\circ}(\Hom(V,S) \setminus \tau^{-1}(CZ_{r-2,n}),\ZZ/2)
    \end{gather*}
    are isomorphisms for $i < 2n-2r+3$ and injective for $i = 2n-2r +3$.
\end{lemma}
\begin{proof}
Combine Lemma \ref{BGlemma2} and Lemma \ref{BGlemma1}.
\end{proof}

\begin{lemma}\label{GMLefschez}
Let $L \subseteq \PP(\Sym^2V^\vee)$ be a generic codimension $c$ linear subspace.
The homomorphisms 
\begin{gather*}
H^i(W^{\mathrm{sm}}_{r,n},\ZZ) \to H^i(L \times_{\PP(\Sym^2 V)^\vee} W^{\mathrm{sm}}_{r,n},\ZZ) \\
H^i(W^{\mathrm{sm}}_{r,n},\ZZ/2) \to H^i(L \times_{\PP(\Sym^2 V^\vee)} W^{\mathrm{sm}}_{r,n},\ZZ/2)
\end{gather*}
are isomorphisms for $i \le \dim W_{r,n}-c$ and injective for $i=\dim W_{r,n}-c$.
\end{lemma}
\begin{proof}

 The generalised Lefschetz theorem of Goresky--MacPherson \cite[Thm p.150]{GM} states that we have isomorphisms on the level of homotopy groups for low degrees.
 Combining this with the Hurewicz theorem gives the statement for cohomology groups.
\end{proof}


\begin{proposition}
\label{thm:compareBGOToX}
Let $L \subseteq \PP(\Sym²V^\vee)$ be a generic codimension $c$ subspace.
The homomorphisms
\begin{gather*}
H^j(\BGO(S)^\circ,\ZZ) \to H^*(L \times_{\PP(\Sym^2 V^\vee)} W_{r,n}^{\mathrm{sm}},\ZZ) \\
H^j(\BGO(S)^\circ,\ZZ/2) \to H^*(L \times_{\PP(\Sym^2 V^\vee)} W_{r,n}^{\mathrm{sm}},\ZZ/2)
\end{gather*}
are isomorphisms for $j < N$ and injective for $j = N$, where
\[
N = \min(2n-1, \dim W_{r,n} - c)
\]
\end{proposition}
\begin{proof}
Combine Lemmas \ref{BGlemma1}, \ref{thm:cohomologyOfComplement} and \ref{GMLefschez}.
\end{proof}


\begin{corollary}
\label{thm:summaryOfHX}
If $c \le 4n-11$, then
\[
H^0(X^{\mathrm{sm}}_{n,c},\ZZ) = \ZZ, H^1(X^{\mathrm{sm}}_{n,c},\ZZ) = 0, H^2(X^{\mathrm{sm}}_{n,c},\ZZ) = \ZZ, H^3(X^{\mathrm{sm}}_{n,c},\ZZ) = \ZZ/2.
\]
If moreover $c \le 4n - 13$, then the square of the non-zero class in $H^3(X^{\mathrm{sm}}_{n,c},\ZZ)$ does not vanish modulo 2.
\end{corollary}
\begin{proof}
Since $X_{n,c} = L \times_{\PP(\Sym^2 V^\vee)} W_{4,n}$ for a generic codimension $c$ subspace $L$, Bertini's theorem implies $X_{n,c}^{\mathrm{sm}} = L \times_{\PP(\Sym^2 V^\vee)} W_{4,n}^{\mathrm{sm}}$.

The first claim then follows from Propositions \ref{prop:cohomologyofGBO} and \ref{thm:compareBGOToX}.
For the second claim, let $\phi \colon H^*(\BGO(4)^\circ, \ZZ) \to H^*(X^{\mathrm{sm}}_{n,c}, \ZZ)$ be the natural homomorphism.
Taking $0 \not= \gamma \in H^3(\BGO(4)^\circ,\ZZ)$, we know from Proposition \ref{prop:cohomologyofGBO} that $\gamma^2 \not= 0 \pmod 2$.
If $c \le 4n-13$, Proposition \ref{thm:compareBGOToX} implies that the map
\[
H^6(\BGO(4)^\circ, \ZZ/2) \to H^6(X^{\mathrm{sm}}_{n,c},\ZZ/2)
\]
is injective, and it follows that $\phi(\gamma)^2 = \phi(\gamma^2) \not= 0 \pmod 2$.
\end{proof}

\begin{remark}\label{remark:Brauer}
When $X_{n,c}$ is smooth and rationally connected, the torsion subgroup of $H^3(X_{n,c},\ZZ)$ 
can be identified with the Brauer group 
$\Br(X_{n,c})$. Under this identification, the generator of $H^3(X_{n,c},\ZZ)\simeq \ZZ/2$ is represented by Brauer-Severi variety given by the restriction of $\eta : U_{4,n}\to W_{4,n}$ to $X_{n,c}$.
\end{remark}

\section{The varieties $X_{n,c}$}\label{specialexamples}
We now analyse a few particularly interesting choices of $n$ and $c$.
\subsection{The case $c = 2n-1$}
Let $X = X_{n, 2n-1}$.
Collecting our work, we have now proved Theorem \ref{mainthm1}, restated more precisely as follows.
\begin{theorem}
\label{thm:mainthm1PreciseVersion}
The variety $X$ is nonsingular of dimension $2n-6$ with $K_X = -H$, and hence Fano.
It has Picard number 1 and $H^3(X,\ZZ) = \ZZ/2$.
\end{theorem}
\begin{proof}
The singular locus in $W_{4,n}$ has dimension $2n-2$ by Proposition~\ref{thm:propertiesOfZ} and Corollary~\ref{singularlocusWr}, so $X$ is nonsingular by Bertini's theorem.
The proof of Lemma \ref{ikkeFano} gives $K_X = -H$.
Finally, $H^3(X,\ZZ)$ is computed in Corollary~\ref{thm:summaryOfHX}.
\end{proof}

\subsection{The Fano fourfold}\label{defineFano4fold}
Specialising further, let $X = X_{5,9}$.
The strata $Z_{i,5}$ were described in Example \ref{ex:Vdim5}. The quintic hypersurface ${\overline Z}_{4,5}\subset\PP^{14}$ parameterizes singular quadrics in  5 variables, i.e., cones over quadrics in $\PP^3$.  The double cover $W_{4,n}\to {\overline Z}_{4,5}$ is ramified along the  $\overline Z_{2,5}\subset \PP^{14}$, which is a singular 8-fold.

\begin{proposition}
    The fourfold $X$ is a Fano variety with invariants
    \begin{enumerate}
    \item $\Pic(X)=\ZZ H$, with $H^4=10$.
        \item $-K_X=H$
        \item $h^{1,3}(X)=9$, $h^{2,2}(X)=67$.
        \item $H^3(X,\ZZ)=\ZZ/2$.
    \end{enumerate}
\end{proposition}
\begin{proof}
(1), (2) and (4) follow from Theorem \ref{thm:mainthm1PreciseVersion}. 
(3) follows from Lemma \ref{surface-hodgenumbers} and Corollary \ref{HPD-hodgenumbers} below.
\end{proof}


\subsubsection{Homological projective duality}
In the paper \cite{Rennemo_2020}, the second named author studies derived categories of linear sections of the stack $\Sym^2 \PP^{n-1}$ from the perspective of homological projective duality \cite{kuznetsovHomologicalProjectiveDuality2007a}.
When $n$ is odd, the paper defines a noncommutative resolution $Y_n$ of $W_{n-1,n}$, and shows that linear sections of this noncommutative resolution are related to dual linear sections of $\Sym^2\PP^{n-1}$ in precisely the way predicted by HP duality, which strongly suggest that $Y_n$ is HP dual to $\Sym^2 \PP^{n-1}$.

Specialising to the case $n = 5$ and linear sections of the appropriate dimensions gives the following result.
Let $V = \CC^5$, let $L_1, \ldots, L_9$ be general hyperplanes in $\PP(\Sym^2 V)$, and let $L$ be their intersection.
Let 
\[
L^\perp = \langle L_1,\ldots, L_9 \rangle \subset \PP(\Sym^2 V^\vee)
\]
be the orthogonal complement.

In this language, $X = L \times_{\PP(\Sym^2 V)} W_{4,5}$.
Since $X$ avoids the singular locus of $W_{4,5}$, the noncommutative resolution $Y_5$ of $W_{4,5}$ is equivalent to $W_{4,5}$, and the main theorem of \cite{Rennemo_2020} applies.
On the other side of the HP duality we find
\begin{equation}
\label{eqn:definitionOfS}
S = L^\perp \times_{\PP(\Sym^2 V^\vee)} \Sym^2\PP(V^\vee),
\end{equation}
which is the intersection of 6 general $(1,1)$-divisors in $\Sym^2 \PP(V^\vee) = \Sym^2 \PP^4$.
The following is a slight amplification of the main result of \cite{Rennemo_2020}.
\begin{proposition}
\label{thm:SODofX}
The category $D(X)$ admits a semiorthogonal decomposition
\[
D(X) = \langle D(S), E_1, E_2, E_3, E_4 \rangle,
\]
where the $E_i$ are exceptional objects.
\end{proposition}
The amplification consists in the fact that \cite{Rennemo_2020} only proves that $D(S)$ includes as a semiorthogonal piece in $D(X)$.
The fact that the orthogonal complement is generated by 4 exceptional objects is not difficult to show using the techniques of the paper.

\begin{lemma}\label{surface-hodgenumbers}
The surface $S$ in \eqref{eqn:definitionOfS} is smooth of degree 35 with respect to the embedding $S \subset \PP(\Sym^2 V^\vee)$.
It has Hodge numbers 
\[
h^{1,0}(S)=0, h^{2,0}(S)=9, h^{1,1}(S)=65.
\]
\end{lemma}
\begin{proof}
    The map $\PP^4\times \PP^4\to \Sym^2(\PP^4)$ induces an étale double cover $\pi:T\to S$ where $T$ is a general complete intersection of 6 symmetric $(1,1)$-divisors in $\PP^4 \times \PP^4$.
    In particular, $T$ is simply connected by the Lefschetz theorem. Furthermore, we find that $K_T=\OO_{\PP^4\times \PP^4}(1,1)|_T$, which gives $\dim H^0(T,K_T) = 19$ and $K_T^2=70$. 
    From Noether's formula, $\chi(\OO_S)=1+h^{2,0}(S)=(1+19)/2=10$, giving $h^{2,0}(S)=9$. As $\pi$ is étale of degree 2, $K_T=\pi^*K_S$,  $K_S^2=35$ and hence $\chi_{\mathrm{top}}(S)=85$ by Noether's formula. From this we find that $h^{1,1}(S)=65$.
\end{proof}

\begin{corollary}\label{HPD-hodgenumbers}
With $S$ and $X$ as above, we have
\[
\sum_{i = 0}^4 h^{i,i}(X) = \sum_{i=0}^2h^{i,i}(S) + 4,
\]
\[
h^{1,2}(X) = h^{0,1}(S)
\]
and
\[
h^{1,3}(X) = h^{0,2}(S)
\]
\end{corollary}
\begin{proof}
The semiorthogonal decomposition in Proposition \ref{thm:SODofX} gives the relation of Hochschild homology groups
\[
HH_*(X) \cong HH_*(S) \oplus \CC^4.
\]
Expressing Hochschild homology via Hodge numbers through
\[
\dim HH_i(-) = \sum_{p-q = i} h^{p,q}(-),
\]
and using the fact that $h^{0,i}(X) = 0$ since $X$ is Fano gives the result.
\end{proof}

\begin{example}\label{fourfoldBrauergroup}
The fact that $\Tors H^3(X,\ZZ)\neq 0$ can be seen as a consequence of the fact that the conic bundle $\eta : U_{4,5}\to W_{4,5}$ does not admit a rational section.

To see this, recall that $U_{4,5}$ is a projective bundle over the Grassmannian $G=\Gr(3,V)$. Explicitly, $U_{4,5}=\PP(E)$ where $E$ is the rank 9 vector bundle appearing as the kernel of the natural map $S^2(V^\vee\otimes \OO_G)\to S^2(\mathcal U^\vee)$, and where $\mathcal U$ is the universal subbundle of rank $3$. Now, if $D\subset \PP(E)$ is the divisor determined by a rational section of $\eta$, $D$ is linearly equivalent to a divisor of the form $aL+bG$, where $L=\OO_{\PP(E)}(1)$ and $G$ is the pullback of $\OO_{\Gr(3,V)}(1)$. We must also have $D\cdot L^{13}=10$ (as the 1-cycle $L^{13}$ is represented by 10 fibers of $\PP(E)\to W_{4,5}$). On the other hand, using the Chern classes of $S^2(\mathcal U^\vee)$, we compute that $D\cdot L^{13}=-20b$, contradicting the condition that $b$ is an integer. 

This shows that the Brauer group of $W_{4,5}^{\mathrm{sm}}$ is non-trivial. In our case, we may identify the Brauer group with $\Tors H^3(W_{4,5}^{\mathrm{sm}},\ZZ)$ because $H^2(W_{4,5}^{\mathrm{sm}},\ZZ)=\ZZ$ is generated by algebraic classes \cite[Proposition 4]{beauville2016luroth}. Finally, Lemma \ref{GMLefschez} shows that $H^3(W^{\mathrm{sm}}_{4,5},\ZZ)\to H^3(X,\ZZ)$ is an isomorphism, so the latter group has non-trivial torsion part as well. 

For an alternative approach to the absence of rational sections, see Claim \ref{c1} in the Appendix.



\end{example}

\subsection{The case $c=2n-2$}\label{isolatedsingcase}
Let $X = X_{n, 2n-2}$.
Then $X$ has dimension $2n-5$, isolated singularities in $\sigma^{-1}(\overline Z_{2,n}) \cap X$, and $K_X = -2H$.
Let $\widetilde X\to X$ be the blow-up at the singular points. 
Then the exceptional divisor $E$ is a disjoint union of components $E_1,\ldots, E_s$, all of which are isomorphic to $\PP^{n-3} \times \PP^{n-3}$, by Proposition~\ref{thm:localStructureOfW}.

By Corollary \ref{thm:summaryOfHX}, we have $H^3(X^{\mathrm{sm}},\ZZ) = \ZZ/2$.
Since $X^{\mathrm{sm}} \simeq \widetilde X - E$, we get a pullback map $H^3(\widetilde X,\ZZ) \to H^3(X^{\mathrm{sm}},\ZZ)$.
This map is an isomorphism by the exact sequence
\begin{equation*}
H^1(E,\ZZ)\to H^3(\widetilde X,\ZZ)\to H^3(X \setminus E,\ZZ)\to H^2(E,\ZZ),
\end{equation*}
using also that $H^1(E,\ZZ) = 0$ and $H^2(E,\ZZ)$ is torsion free.
\begin{proposition}
\label{thm:higherDimArtinMumford}
    For each $n\ge 4$, $\widetilde X$ is a smooth  projective variety of dimension $2n-5$ with $$\Tors H^3(\widetilde X,\ZZ)\neq 0.$$ 
    The variety $\widetilde X$ is unirational, but not stably rational.
\end{proposition}
\begin{proof}
Only the unirationality remains to be proved.
The incidence variety $U_{4,n}$ of \eqref{defineUandW} is a $\PP^{2n-2}$-bundle over the Grassmannian $\Gr(n-2,V)$. 
This means that if $X$ is a complete intersection of $2n-2$ divisors in $W_{4,n}$, the preimage $U_X=\eta^{-1}(X)$ is birational to $\Gr(n-2,V)$. 
Therefore $U_X$ is rational, and hence $X$ is unirational.
\end{proof}

\begin{example}
When $n=4$, $X$ is a double cover of $\PP^3$ branched along a singular quartic surface.
This is the example famously studied by Artin and Mumford in \cite{artinmumford}, and for which they prove Proposition~\ref{thm:higherDimArtinMumford}.
Here $X$ has 10 ordinary double points and the blow-up $\widetilde X$ contains 10 exceptional divisors isomorphic to $\PP^1\times \PP^1$.
\end{example}


\subsection{The case $n=4$, $c < 6$}\label{AMcase}
The Artin--Mumford examples of $X_{4,6}$ can also naturally be generalised to $X_{4,c}$ with $c < 6$.
We will explain that, at least when $c = 4$ or $5$, these do not have torsion in $H^3$ in their smooth models (correcting a claim made in a MathOverflow answer \cite{mathoverflow}).

The singular locus of $X_{4,c}$ has codimension 3 and is a smooth Enriques surface or a smooth genus 6 curve when $c = 4$ and $c= 5$, respectively.
There is a resolution $\pi \colon \widetilde X \to X_{4,c}$ obtained by blowing up the singular locus, where the exceptional divisor is a $\PP^1 \times \PP^1$-bundle over the singular locus.

\begin{proposition}
With $\widetilde X$ as above, we have that the group $H^3(\widetilde X,\ZZ)$ is torsion free for $c = 5$ and 0 for $c = 4$.
\end{proposition}
\begin{proof}
To show that $H^3(\widetilde X,\ZZ)$ is torsion free, we first remark that $H^3(X,\ZZ)$ has no torsion by Corollary \ref{Lefschetzdoublecovers} below. 
Next, we consider the Leray spectral sequence associated to the blow-up $\pi : \widetilde{X}\to X$, with $E_2$-page $H^p(X,R^q\pi_*\ZZ)$ converging to $H^{p+q}(\widetilde{X},\ZZ)$.
Let $S \subset X$ be the singular locus.
We have $R^0\pi_* \ZZ_{\widetilde X} = \ZZ_X$, $R^1\pi_*\ZZ = 0$, $R^2\pi_*\ZZ_{\widetilde X} = F$ and $R^3\pi_*\ZZ = 0$, where $F$ is a rank two local system.
More explicitly, we have $F = R^2\pi_*\ZZ_{E}$, and since $E$ is $\PP^1 \times \PP^1$-bundle over $S$, this means $F \cong R^0f_*\ZZ_{S'}$, where $f \colon S' \to S$ is the étale double cover of $S$ corresponding to the two families of lines in each fibre of $E \to S$.

By Corollary \ref{Lefschetzdoublecovers}, we have $H^3(X,\ZZ) = 0$, and so the only non-vanishing term of the $E_2$-page of the spectral sequence is
\[
H^1(X,R^2\pi_*\ZZ_{\widetilde X}) \cong H^1(S,R_0f_*\ZZ_{S'}) = H^1(S', \ZZ).
\]
Running the spectral sequence then gives
\[
0 \to H^3(\widetilde X, \ZZ) \to H^1(S',\ZZ) \to H^4(X,\ZZ)
\]
Since $H^1(S',\ZZ)$ is torsion free, the same is true for $H^3(\widetilde X,\ZZ)$.
When $c = 4$, the variety $S$ is an Enriques surface, so that $S'$ is either a K3 surface or two copies of $S$; in either case $H^1(S',\ZZ) = 0$ which gives $H^3(\widetilde X,\ZZ) = 0$.
\end{proof}

In the argument above, we used the following version of the Weak Lefschetz hyperplane theorem for singular varieties.

\begin{proposition}\label{BMLefschetz}
Let $V$ be a projective variety of dimension $n+1$ and let $D$ be an ample divisor which is disjoint from the singular locus $\sing(V)$. Then the natural maps
$$
H_i(D,\ZZ)\langto H^{2n+2-i}(V,\ZZ)
$$are isomorphisms for $i<n$ and surjective for $i=n$.
\end{proposition}
\begin{proof}
Letting $U=V-D$, the relative cohomology sequence takes the form
$$
\cdots \to H^i(V,V-D,\ZZ)\to H^i(V,\ZZ)\to H^i(U,\ZZ)\to H^{i+1}(V,V-D,\ZZ)\to \cdots
$$Moreover, since $V$ is smooth in a neighborhood of $D$, we may identify $H^i(V,V-D,\ZZ)$ with $H_{2n+2-i}^{BM}(D)=H_{2n+2-i}(D)$. Now, using that $U$ is affine of dimension $n+1$, the cohomology groups $H^i(U,\ZZ)$ vanish for all $i>n+1$, by Artin's vanishing theorem. 
\end{proof}
\begin{corollary}\label{Lefschetzdoublecovers}
    Let $\sigma: X\to \PP^n$ be a ramified double cover. Then for each $i < \frac{n}2$:
\begin{enumerate}[(i)]
    \item $H_{2i}(X,\ZZ)=\ZZ$ and $H_{2i-1}(X,\ZZ)=0$ 
        \item $H^{2i}(X,\ZZ)=\ZZ$ and $H^{2i-1}(X,\ZZ)=0$
\end{enumerate}
\end{corollary}
\begin{proof}
    Note that $X$ can be defined by an equation of the form $z^2=f(x_0,\ldots,x_n)$ in the weighted projective space $V=\PP(1,\ldots,1,\tfrac{d}{2})$. Thus $X$ is an ample divisor, disjoint from the one singular point of $V$. Thus the conditions of Proposition~\ref{BMLefschetz} hold, and we find that $H_j(X,\ZZ) = H^{2n-j}(V,\ZZ)$ when $j < n$.
    The cohomology of $V$ is computed in \cite[Theorem~1]{kawasaki_cohomology_1973}, which gives claim (i), and claim (ii) follows by the Universal Coefficient theorem.
\end{proof}

\section{Proof of Theorem \ref{mainthm2}}
\label{coniveausection}
In this section we state and prove a precise version of Theorem \ref{mainthm2}.
We first recall some general background on the coniveau filtrations on cohomology of algebraic varieties, referring to \cite{Benoist_Ottem_2021} for details.
 We restrict ourselves to the case of cohomology with integral coefficients $H^i(X,\ZZ)$ on a smooth projective variety $X$ over $\CC$.

A cohomology class $\alpha \in H^l(X,\ZZ)$ is said to be of {\em coniveau} $\ge c$ if it restricts to 0 on $X-Z$ where  $Z$ is a closed subset of codimension at least $c$ in $X$. 
These classes give the {\em coniveau filtration}
$ N^c H^l(X,\ZZ) \subset H^l(X,\ZZ).$ 
Equivalently, viewing $H^l(X,\ZZ)$ as $H_{2n-l}(X,\ZZ)$ via Poincaré duality, a class $\alpha\in H_{2n-l}(X,\ZZ)$ is of coniveau $\ge c$ if and only if $\alpha = j_*\beta$ for some $\beta\in H_{2n-l}(Y,\ZZ)$, where $j:Y\to X$ is the inclusion of a closed algebraic subset of $X$ of codimension at least $c$. 
So for example, $N^cH^{2c}(X,\ZZ)$ consists of exactly the algebraic classes in $H^{2c}(X,\ZZ)$.

A class $\alpha\in H^l(X,\ZZ)$ is said to be of {\em strong coniveau} $\ge c$ if 
$
\alpha=f_*\beta
$ where $f:Z\to X$ is a proper morphism, $Z$ is a smooth complex variety of dimension at most $n-c$, and $\beta\in H^{l-2c}(Z,\ZZ)$. 
Equivalently, $\alpha$ has strong coniveau $\ge 1$ if $\alpha=\tilde j_*\beta$ is the Gysin pushforward of a class $\beta\in H^{*}(\widetilde Y,\ZZ)$ and $\tilde j : \widetilde Y\to Y$ is the desingularization of a closed subset of codimension $\ge c$. These classes give the {strong coniveau filtration} $
\widetilde N^c H^l(X,\ZZ) \subset H^l(X,\ZZ)
$. 

We have $\widetilde N^c H^l(X,\ZZ)\subset N^c H^l(X,\ZZ)$ for every $c$. Moreover, the quotient $$N^1 H^l(X,\ZZ)/ \widetilde N^1 H^l(X,\ZZ)$$ is a birational invariant among smooth projective varieties \cite{Benoist_Ottem_2021}. This invariant is particularly interesting for rationally connected varieties $X$. In this case, all cohomology classes are of coniveau $\ge 1$:

\begin{proposition}\label{ctvRC}
    Let $X$ be a rationally connected smooth projective complex variety. Then for any $l>0$,
$$
N^1H^l(X,\ZZ)=H^l(X,\ZZ).
$$
\end{proposition}\begin{proof}
 See \cite{blochsrinivas} for the case $\ell=3$, and \cite{CTV} in general. 
\end{proof}





In \cite[Question 3.1]{voisinconiveauRC}, Voisin asked whether $\widetilde N^1H^l(X,\ZZ)=N^1H^l(X,\ZZ)$ for $X$ a rationally connected variety, i.e., whether all cohomology classes are of strong coniveau 1 (see also \cite[Section 7.2]{Benoist_Ottem_2021}). In the same paper, she proved that any class in $H^3(X,\ZZ)$ modulo torsion is of strong coniveau 1. This was extended by 
Tian \cite[Theorem 1.23]{tian2022local} who proved that $H^3(X,\ZZ)=\widetilde{N}^1H^3(X,\ZZ)$ for any rationally connected threefold. 
Our Fano varieties give the first rationally connected examples where the two coniveau filtrations are different.

In \cite{Benoist_Ottem_2021}, the following topological obstruction to strong coniveau $\ge 1$ was introduced.
\begin{proposition}\label{maintopologicalobstruction}
    If  $\alpha\in H^3(X,\ZZ)$ is a class of strong coniveau $\ge 1$, then the mod 2 reduction $\bar \alpha \in H^3(X,\ZZ/2)$ satisfies $$\bar \alpha^2= 0 \mbox{ in }H^6(X,\ZZ/2).$$
\end{proposition}

\begin{proof}
    This is a special case of \cite[Proposition 3.5]{Benoist_Ottem_2021}.
\end{proof}
Here is the precise version of Theorem \ref{mainthm2}:

\begin{theorem}
\label{thm:mainThm2PreciseVersion}
For $n \ge 6$, the variety $X_{n, 2n-1}$ from Definition \ref{defn:CompleteIntersectionX} is a Fano variety of dimension $2n-6$ with $K_X = -H$, such that
\[
0 = \widetilde N^1H^3(X,\ZZ) \not= N^1H^3(X,\ZZ)  =H^3(X,\ZZ) \cong \ZZ/2.
\]
\end{theorem}
\begin{proof}
Let $X = X_{n,2n-1}$. 
The computation of $\dim X$, $H^3(X,\ZZ)$ and $K_X$ is part of Theorem~\ref{thm:mainthm1PreciseVersion}.
Since $X$ is Fano, it is rationally connected, so Proposition~\ref{ctvRC} gives $N^1H^3(X,\ZZ) = H^3(X,\ZZ)$.
Corollary~\ref{thm:summaryOfHX} shows that the class $\alpha \not= 0 \in H^3(X,\ZZ)$ is such that the mod 2 reduction of $\alpha^2$ is non-zero.
Proposition~\ref{maintopologicalobstruction} then implies $\alpha \not\in \widetilde N^1H^3(X,\ZZ)$, so $\widetilde N^1H^3(X,\ZZ) = 0$.
\end{proof}





\begin{remark}
    We can obtain examples of other rationally connected varieties where $\widetilde N^cH^{l}\neq  N^cH^{l}$ for any $c\ge 1$ and $l\ge 2c+1$ by taking appropriate products with projective spaces (see e.g., \cite[Theorem 4.3]{Benoist_Ottem_2021}). 
\end{remark}

\begin{remark}[The Artin--Mumford example]
In light of Theorem \ref{mainthm2}, it is natural to ask whether the 2-torsion class $\alpha\in H^3(X,\ZZ)$ in the Artin--Mumford example has strong coniveau $\ge 1$, i.e., whether the birational invariant \eqref{coniveauquotient} is zero. It turns out that this is indeed the case: Inspecting Artin--Mumford's `brutal procedure' in \cite[p. 82--83]{artinmumford}, shows that the class $\alpha$ is obtained from a cylinder map $H^1(C,\ZZ)\to H^3(X,\ZZ)$ from an elliptic curve $C$. In other words, $\alpha$ is the pushfoward from a class in $H^1$ from some ruled surface $S$ over $C$.

Note that this can also be seen as a special case of \cite[Theorem 1.23]{tian2022local}.
\end{remark}

\subsection{Open questions}We conclude with two open questions regarding the two coniveau filtrations:
\begin{question}
    Are there rationally connected varieties $X$  with  $\widetilde N^1H^l(X,\ZZ) \neq N^1H^l(X,\ZZ)$ for some $l>0$ and torsion free $H^l(X,\ZZ)$?
\end{question}
\begin{question}
    Are there rationally connected varieties of dimension 4 or 5 where $\widetilde N^1H^l(X,\ZZ) \neq N^1H^l(X,\ZZ)$ for some $l>0$?
\end{question}

    \begin{remark}
 Let $X = X_{5,9}$ be the fourfold from Section \ref{defineFano4fold}.
 Then we don't know if the generator $\alpha$ of $H^3(X,\ZZ)$ has strong coniveau $\ge 1$. 
 We can show, however, that $\overline \alpha^2=0$ in $H^6(X,\ZZ/2)$, so the topological obstruction of Proposition \ref{maintopologicalobstruction} vanishes. To see this, we use the fact that the third integral Steenrod square $\operatorname{Sq}_{\ZZ}^3 \colon H^p(Z,\ZZ) \to H^{p+3}(X,\ZZ)$ is naturally identified with the third differential $d_3$ in the Atiyah--Hirzebruch spectral sequence of topological K-theory, with $E_2$-page $$H^p(X,K^{q}(pt))=\begin{cases}H^p(X,\ZZ) & q \text{ even }\\ 0 & \text{otherwise} \end{cases}$$ converging to $K^{p+q}(X)$. Now $H^*(X,\mathbb{Z})$ has torsion only in degrees $3$ and $6$, with torsion part $\mathbb{Z}/2$ in each of these degrees. It also has torsion $\mathbb{Z}/2 \oplus \mathbb{Z}/2$ in its topological $K$-theory, by Proposition \ref{thm:SODofX} (because $S$ is a general type surface with fundamental group $\mathbb{Z}/2$). 
 This implies $d_3 = 0$, since otherwise the Atiyah-Hirzebruch spectral sequence would give that the topological $K$-theory of $X$ was torsion free.
 Since $\operatorname{Sq}_{\ZZ}^3$ is an integral lift of the usual (mod 2) Steenrod square $\operatorname{Sq}^3$, we then get 
 \[
 \overline \alpha \cup \overline \alpha = \operatorname{Sq}^3(\overline \alpha)=0.
 \]
\end{remark}In general, it would be interesting to find other obstructions to the equality of the two coniveau filtrations than the topological obstructions of \cite{Benoist_Ottem_2021}.




\appendix

\section{Comments on the Fano fourfold\texorpdfstring{\\\smallskip}{ (}by {J\'anos Koll\'ar}\texorpdfstring{}{)}}

\setcounter{footnote}{0}




In the present paper, Ottem and Rennemo construct smooth Fano fourfolds
$X$ such that $H_2(X, \z)\cong \z+\z/2$. This appendix gives a shorter computation of $H_2(X, \z)$, see Claims~\ref{c1} and~\ref{c2-1}.

We also add two new results. In  Claim~\ref{c2X} we exhibit two lines $L', L''\subset X$ such that  $L'- L''$ generates the  torsion in  $H_2(X, \z)\cong \z+\z/2$.

Then 
in Paragraph~\ref{s15} we show that $X$ is birational to a double cover of $\p^4$ ramified along a degree 18 hypersurface $R$, which is obtained as the 5-secants of a degree 15, smooth,  determinantal surface
$
S=\bigl(\rank N\leq 4\bigr)\subset \p^4,
$
where $N$ is a  $6\times 5$  matrix
whose entries are linear forms.
Although $S$ is smooth, it  is not a general determinantal surface, since
the latter have only  1-parameter families of 5-secants.

The higher dimensional examples constructed in the paper can  also be treated with minor changes.

We refer to \cite[Chapters VIII-IX]{room1938geometry} for  symmetric determinantal varieties and to \cite{kronecker1891algebraische} for the classification of lines on them. 
\medskip

\begin{say}[Basic set-up]\label{s1} We recall the construction of the Fano fourfolds in the paper. This is the case when $r=4$ and $n=5$ (see Section \ref{defineFano4fold}). Let $Z_i\subset \p^{14}$ be the space of rank $\leq i$ quadrics in $\p^4$.
Our main interest is 
$Z=Z_4\subset \p^{14}$, the space of rank $\leq 4$ quadrics in $\p^4$.
It is a quintic hypersurface.

The universal deformation of a rank 3 quadric is given by
$$
x_0^2-x_1^2-x_2^2+z_1x_3^2+z_2x_3x_4+z_3x_4^2=0.
\eqno{(\ref{s1}.1)}
$$
This has rank $\leq 4$ iff $z_2^2-4z_1z_3=0$. 
So $Z$   is singular along
$Z_3$,   with transversal singularity type $A_1$.  


As in \eqref{eqn:DefinitionOfIncidenceVariety}, define $U$ to be the  space of pairs  $(L^2\subset Q)$ where $L^2\subset\p^4$ is a 2-plane and $Q\subset\p^4$ a quadric. We have projections
$$
G\stackrel{p_1}{\longleftarrow} U \stackrel{p_2}{\longrightarrow} Z,
$$
and $\pic(U)\cong \z^2$ is generated by $p_1^*\o_G(1)$ and $p_2^*\o_{Z}(1)$. 

Let $U_2\subset U$ be the subset where the quadric has rank $\leq 2$.
These quadrics split into 2 hyperplanes, one of which must contain $L^2$. 
So $U_2$  is a $\p^1\times \p^4$-bundle over $\grass(2,4)$.
Set $U^\circ:=U\setminus U_2$. This is the preimage of the open set $Z^\circ:= Z_4\setminus Z_2$.

Since $U_2\subset U$ has codimension 3, 
$H^i(U^\circ, \z)=H^i(U, \z)$ for $i\leq 4=2(3-1)$.
In particular, 
$$
H^0(U^\circ, \z)= \z, \ 
H^1(U^\circ, \z)= 0, \ 
H^2(U^\circ, \z)= \z^2, \ 
H^3(U^\circ, \z)= 0.
$$
As in Section \ref{sec:doublecovers}, we define $W$ using the  Stein factorization of $p_2:U\to Z$
$$
U\stackrel{\eta}{\longrightarrow} W \stackrel{\sigma}{\longrightarrow}Z.
$$
In the coordinates (\ref{s1}.1),  $Z\cong (z_2^2-4z_1z_3=0)\times \a^{11}$.
Set $w_1:=\sqrt{z_1}$ and $w_3:=\sqrt{z_3}$.  Since $w_1w_3=z_2/2$, adjoining both $w_1, w_2$ is a degree 2 covering and 
$$
z_1x_3^2+z_2x_3x_4+z_3x_4^2=(w_1x_3+w_3x_4)^2.
$$
Thus, locally analytically over the points of
$Z_3\setminus Z_2$,  we have 
$W\cong  \a^2_{w_1, w_2}\times \a^{11}$, and the family of quadrics becomes
$$
(x_0^2-x_1^2-x_2^2+(w_1x_3+w_2x_4)^2=0)\times \a^{11}.
\eqno{(\ref{s1}.2)}
$$
Thus $U$ is the family of 2-planes in the same family as
$$
\bigl(x_0-x_1=x_2-(w_1x_3+w_2x_4)=0\bigr).
$$
Each of these has a unique intersection point with
$(x_3=x_4=0)$. Thus   $U$  is locally analytically
isomorphic to the trivial family
$$
W\times (x_0^2-x_1^2-x_2^2=0)\subset W\times \p^2.
$$
Therefore, 
restricting to $U^\circ $ we get 
$$
U^\circ\stackrel{\eta^\circ}{\longrightarrow} W^\circ \stackrel{\sigma^\circ}{\longrightarrow}Z^\circ,
$$
and   $\eta^\circ: U^\circ\to W^\circ$ is a smooth morphism with conics as fibers.
\end{say}

\begin{claim}\label{c1} $H^3(W^\circ, \z)\cong \z/2$ and $\eta$ has  no rational sections.
\end{claim}

 \begin{proof}Let $C\subset U^\circ$ be a fiber of $\pi^\circ$ over a rank 3 conic.
We can think of $C$ as (cones over) lines on a  quadric cone, or after further degeneration, as  (cones over) 2 pencils of lines on 2 planes.
So $C$ is 2-times the class of (cones over) a  pencil of lines in a plane.
Thus the image of $H^2\bigl(  U^\circ, \z\bigr)\to H^2(C, \z)\cong \z$ is twice the generator.
(Note that this splitting of $C$ into 2 components happens in the fibers over $Z_2$, thus outside  $Z^\circ$.)

In the Leray spectral sequence for $\pi^\circ$, the only interesting map is on the $E_3$ page:
$$
\z \cong H^0\bigl(  W^\circ, R^2\pi^\circ_*\z\bigr) \to H^3\bigl(  W^\circ, \z\bigr).
$$
As we noted, the kernel is $2\z$ and $ H^3\bigl(  U^\circ, \z\bigr)=0$.
Thus $ H^3\bigl(  W^\circ, \z\bigr)\cong \z/2$  and $H^2\bigl(  W^\circ, \z\bigr)\cong \z$.  \end{proof}
\medskip

\begin{say}[Construction of $X$] \label{s3}

As in Section \ref{defineFano4fold},
let $X_Z\subset Z$ be the complete intersection of $9$ general hyperplanes and  $X:=X_W\subset W$ its preimage.
Then $X\subset W^\circ$ is a Fano fourfold with $K_X\sim -\sigma^*H$.
\end{say}

\begin{claim}\label{c2-1} $H_2(X, \z)\cong \z+\z/2$.
\end{claim}

\begin{proof}
 The isomorphism $H_2(X, \z)\cong H_2(W^\circ, \z)$ follows from
 the Lefschetz hyperplane theorem of \cite{GM} (see Section \ref{defineFano4fold}).
\end{proof}

\medskip

Note that by \cite{voisin2006integral}, $H_2(X, \z)$ is generated by algebraic curves.
Next we write down a difference of 2 smooth, degree 1 rational curves that generates
the $\z/2$-summand of $H_2(X, \z)$.
\medskip

\begin{claim}\label{c2}  Let $L\subset Z\setminus Z_3\subset \p^{14}$ be a line.
  Its preimage in $W$ is a pair of lines $L'\cup L''$ such that
  \begin{enumerate}
\item  $L'$ and $L''$ are numerically equivalent, 
  \item  $U\times_WL'$ is the ruled surface $\f_1\cong B_p\p^2$,
  \item  $U\times_WL''$ is the ruled surface $\f_0\cong \p^1\times \p^1$,  and
    \item  $L'-L''$  is a  generator of
      the $\z/2$-summand  of $H_2(W^\circ, \z)\cong \z+\z/2$.
      \end{enumerate}
\end{claim}

By Paragraph~\ref{lines.XZ.say}, 
$X_Z$   contains a
2-parameter family of lines, and  Claim~\ref{c2} applies to them.
Thus we obtain the following.

\begin{claim}\label{c2X}  Let $L\subset X_Z\setminus Z_3$ be a line.
Its preimage in $X$ is a pair of lines $L'\cup L''$, and
$L'-L''$  is a  generator of
the $\z/2$-summand  of $H_2(X, \z)\cong \z+\z/2$. \qed
\end{claim}

\begin{say}[Beginning of the proof of Claim~\ref{c2}]
By Claim~\ref{c2-1}  $H_2(X, \z)/(\mbox{torsion})$ is  generated by $K_X\sim -\sigma^*H$,  
 so $L'$ and $L''$ are numerically equivalent.

  By Paragraph~\ref{lines.say},   in suitable coordinates  we can write
$L$ as a family of quadrics
$$
Q(\lambda{:}\mu):=\bigl(x_0(\lambda x_2-\mu x_3)=x_1(\mu x_4-\lambda x_3)\bigr).
$$
All of these contain the 2-plane  $(x_0=x_1=0)$,  defining  a section
$s_0:L\to U$.

The preimage of $L$ in $W$ is a disjoint union of 2 lines
$L'\cup L''$.  We choose $L'$ to be
$\pi\circ s_0:L\to U\to W$.

For any nonzero linear form  $\ell=a\lambda+b\mu $, a section  $C'(\ell)$ of
$\pi: U\to W$ over $L'$ is given by
$$
\ell x_0=\mu x_4-\lambda x_3 \qtq{and} \lambda x_2-\mu x_3=\ell x_1.
$$
For $\ell_1\neq \ell_2$ the 2 sections  $C'(\ell_1), C'(\ell_2)$ meet at the point where $\ell_1=\ell_2$.
Thus  $U\times_WL'$ is the ruled surface  $\f_1$.

In the other family of 2-planes, we have sections  $C''(\ell)$ given by
$$
cx_0=x_1 \qtq{and} \mu x_4-\lambda x_3=c(\lambda x_2-\mu x_3).
$$
These are disjoint for $c_1\neq c_2$. Thus $U\times_WL''$ is the the trivial $\p^1$-bundle.

These show  Claim~\ref{c2}.2--3. \qed
\end{say}

Claim~\ref{c2}.4 is a formal consequence of (\ref{c2}.1--3).
To see this, 
 we need to  discuss how to detect 2-torsion in $H_2$ using
$\p^1$-bundles.  (Similarly, $n$-torsion can be detected using
$\p^{n-1}$-bundles.)

\begin{say}[Comments on $\p^1$-bundles]\label{com1}
Let $X$ be a  normal,  proper variety and 
$\pi:Y\to X$ a $\p^1$-bundle  (\'etale locally trivial).
For a smooth curve $C\to X$, let    $C_Y\to Y_C:=C\times_XY$ be a lifting. Set
$$
\sigma_Y(C):=(C_Y\cdot K_{Y_C/C}) \mod 2.
\eqno{(\ref{com1}.1)}
$$
This is well defined as a function on $A_1(X)$, the group of 1-cycles modulo algebraic equivalence.

If $X$ is smooth and $Y\to X$ has  a rational section  $S\subset Y$, then
$ K_{Y/X}\sim -2S+\pi^*D$ for some Cartier divisor $D$ on $X$. In this case
$$
\sigma_Y(C)\equiv (C\cdot D) \mod 2.
$$
Conversely, assume that  we are over $\c$,  $H_2(X,\z)$ is generated by algebraic curves, and $ \sigma_Y(C)\equiv (C\cdot D) \mod 2$  for every $C$. Then 
$K_{Y/X}-\pi^*D$ is divisible by 2, giving a rational section.
In any case, we get the following.
\end{say}

\begin{claim}\label{c3} If there is a numerically trivial 1-cycle $C$ such that
$\sigma_Y(C)\equiv 1 \mod 2$, then
$[C]$ is a nontrivial torsion element in  $H_2(X,\z)$,
and $Y\to X$ has no rational sections. \qed
\end{claim}

\begin{say}[End of the proof of Claim~\ref{c2}]
  Since   $U\times_WL'\cong \f_1$, (\ref{com1}.1)   shows that
   $\sigma_U(L')\equiv 1 \mod 2$. Similarly, 
  $U\times_WL''\cong \f_0$ implies that 
  $\sigma_U(L'')\equiv 0 \mod 2$.
  Thus $C:=L'-L''$ is numerically trivial and
   $\sigma_U(C)\equiv 1 \mod 2$.
We can now apply Claim~\ref{c3}. \qed
\medskip

In both cases we could have used the isomorphism
$$
\omega_{U/W}\cong p_1^*\o_{G}(-1)\otimes p_2^*\o_{Z}(1)
$$
to compute $\sigma_U(L')$ and $\sigma_U(L'')$. \qed
\end{say}

\begin{say}[Lines on $Z$]\label{lines.say} By  \cite{kronecker1891algebraische}, the lines on $Z\setminus Z_2$ form 3 families of dimension 20 each. These are the following.

\medskip
(\ref{lines.say}.1) $\langle Q_1, Q_2\rangle$   where the $Q_i$ contain a common 2-plane. The general such line $L$ is disjoint from $Z_3$, and its preimage in $W$ is a pair of disjoint lines $L'\cup L''$. After coordinate change, these can be written as
$$
x_0(\lambda x_2-\mu x_3)=x_1(\mu x_4-\lambda x_3).
$$

\medskip
(\ref{lines.say}.2)  $\langle Q_1, Q_2\rangle$   where the $Q_i$ have a common singular point. After coordinate change, these can be written as
$$
\lambda q_1(x_1, \dots, x_4)+\mu q_2(x_1, \dots, x_4)=0,
$$
where the $q_i$ are quadratic forms.
The general such line $L$ intersects $Z_3$ at 4 points,  and its preimage in $W$ is a smooth, elliptic curve of degree 2.

\medskip
(\ref{lines.say}.3)  $\langle Q_1, Q_2\rangle$   where the $Q_i$ have rank 2 and $\sing Q_i$ is tangent to $Q_{3-i}$.  The general such line $L$ intersects $Z_3$ at 2 points,  and its preimage in $W$ is a smooth, rational curve of degree 2. After coordinate change, these can be written as
$$
\lambda(x_0^2+x_1x_2)+\mu(x_2x_3+x_4^2)=0.
$$
\end{say}

\begin{say}[Lines on $X_Z$]\label{lines.XZ.say}
The space of lines in $Z\setminus Z_3$ has dimension 20, and with each hyperplane section the dimension drops by 2. So the lines on $X_Z$ form 3 families of dimension 2 each. Only (\ref{lines.say}.1) contains lines that are disjoint from $\sing X_Z$. 

Since there are no lines on $Z_3\setminus Z_2$, 
the only common liness to any 2 of these families are the finitely many double tangents of $\sing X_Z$. 
\end{say}

\begin{say}[Another representation of $X_Z$]\label{s10}
  Let $P_5$ be a general 5-dimensional linear system of quadrics on $P_4:=\p^4$.
For $i=4,5$, we have the projections $\pi_i:P_4\times P_5\to P_i$. For brevity let us write  $(a,b):=a\pi_4^*H_4+ b \pi_5^*H_5$, where
 $H_i$ is the hyperplane class on $P_i$.
  Set
  $$
  Y:=\bigl\{(p, Q): p\in P_4, Q\in P_5, p\in \sing Q\bigr\}\subset P_4\times P_5.
  $$
  The condition $ p\in \sing Q$ is equivalent to the partial derivatives of the equation of $Q$ vanishing at $p$. Thus $Y\subset P_4\times P_5$ is the complete intersection of 5 divisors of bidegree $(1,1)$. Write these as
  $$
  \tsum_{i=0}^4\tsum_{j=0}^5 a^\ell_{ij}y_ix_j\qtq{for} \ell=1,\dots, 5.
  \eqno{(\ref{s10}.1)}
  $$
  Over $P_5$, (\ref{s10}.1) is equivalent to a $5\times 5$ symmetric matrix
  whose entries are the linear forms  $m_i^\ell=\tsum_{j=0}^5 a^\ell_{ij}x_j$.
  The condition $\det (m_i^\ell)=0$ defines  $X_Z$ as in Paragraph~\ref{s3}.

  Over $P_4$, (\ref{s10}.1) is equivalent to a $6\times 5$  matrix
  whose entries are the linear forms  $n_j^\ell=\tsum_{i=0}^4 a^\ell_{ij}y_i$.
  Note that $\pi_4:Y\to P_4$ is birational. Its inverse is  the blow-up
  of a surface\footnote{This is not the surface $S$ of Section \ref{defineFano4fold}.} 
  $$
  S\subset P_4,\qtq{defined by }\rank(n_j^\ell)\leq 4.
  $$ Let $E_4\subset Y$ denote the exceptional divisor.
   $Y$ defines a rational map $P_4\map X_Z\subset P_5$, which is given by the
  $5\times 5$ subdeterminants of $(n_j^\ell)$. Thus
  $E_4\sim (5,-1)|_Y$.

  The inverse rational map  $X_Z\map P_4$ is a bit harder to see.
  It is given by a linear system of divisors as follows.
  Let $H\in |H_4|$ be a hyperplane and set
  $$
  D_H:=\bigl\{Q\in X_Z: H\cap \sing Q\neq\emptyset\bigr\}\subset X_Z.
  $$
  Note that the condition $H\cap \sing Q$ is equivalent to
  $Q|_H$ being singular. (Here we need that $Q$ itself is singular.)
  This gives us the equation  $\det (Q|_H)=0$ for $D_H$. It has degree 4.

  We claim that the intersection of $\left(\det (Q|_H)=0\right)$ with $X_Z$ has multiplicity 2.
  To see this, choose coordinates such that  $H=(x_0=0)$ and
  $Q=(x_0^2+x_2^2+x_3^2+x_4^2=0)$. For its deformations we can make linear coordinate changes to the $x_2, x_3, x_4$, but $x_0$  can only be multiplied by a constant. Thus we get a miniversal deformation family
  $$
  \bigl(x_0^2+t_1x_0x_1+t_2x_1^2+x_2^2+x_3^2+x_4^2=0\bigr)\subset 
P_5\times \a^2_{t_1, t_2}.
  $$
  For a given $t_1, t_2$, the quadric has rank $4$ iff $t_1^2-4t_2=0$, and the singular point is on $(x_0=0)$ iff $t_2=0$.  Their intersection is the length 2 scheme $(t_1^2=0)$.

  Thus the $D_H\subset  X_Z$ have degree $10=\frac12 (4\cdot 5)$ and $2D_H\sim 4H_5|_Y$.
  In particular, the divisor class $D_H-2H_5|_Y$ is 2-torsion in the class group $\cl(X_Z)$.
  The corresponding double cover is our  $X$, constructed  in Paragraph~\ref{s3}.

  Let $E_5\subset Y$ denote the exceptional divisor of $\pi_5:Y\to X_Z$.
  The previous computations suggest that it should be linearly equivalent to $(-1,2)$. However, $X_Z$ has multiplicity 2 along the base locus of
  $|D_H|$, so the correct bidegree is $(-2,4)$.

On $P_4$, the 3 families of  lines  (\ref{lines.say}.1--3)  correspond to

(\ref{s10}.1) conics that are 9-secants of $S$,

(\ref{s10}.2) fibers of $E_4\to S$, and 

(\ref{s10}.3) lines that are 4-secants of $S$.

\end{say}

\begin{say}[$X$ as a double $\p^4$]\label{s15}
  By the previous description, $X$ is birational to a double cover of $\p^4$ ramified along the hypersurface  $R:=\pi_4(E_5)\subset P_4$.

  The degree of $R$ is given by  $(1,0)^3[E_5](1,1)^5$, which works out to be 18.
  The degree of the surface $S$ is $(1,0)^2[E_4](0,1)(1,1)^5=15$.
  Note that $S$ is a $6\times 5$ determinantal surface. However, it is not general since we have a symmetry condition on the  $P_5$ side, so  results about general  determinantal surfaces do not apply to $S$.

  The multiplicity of $Y$ along $S$ is 4. This follows from the computation
  $$
  (1,0)^2[E_4][E_5](1,1)^5=60 = 4\cdot \deg S.
  $$
Thus $R$ is in the 4th symbolic power of the homogeneous ideal of $S$, but not in its 4th power. For general  determinantal surfaces these  are equal by \cite{deconcini1980young}.

  Another interesting property of $S$ is that the fibers of $E_5\to \sing X_Z$
  give 5-secants of $S$. Thus $S$ has a 2-parameter family of 5-secants.
Note that  by (\ref{s10}.3), the family of  4-secants   has  dimension 2 as well.

  Most surfaces in $\p^4$, including  general $6\times 5$ determinantal surfaces,  have only  1-parameter families of 5-secants.

  \end{say}

\begin{ack}  I thank  J.~Ottem, C.~Raicu and B.~Totaro   for many   useful comments.
Partial  financial support    was provided  by  the NSF under grant number
DMS-1901855.
\end{ack}

\bibliographystyle{abbrv}

\end{document}